\documentclass{amsart}
\usepackage{color}
\usepackage{hyperref}
\usepackage{amssymb}
\usepackage{amsmath}
\usepackage[T1]{fontenc}
\usepackage{mathtools}

\newtheorem{thm}{Theorem}

\newtheorem{cor}{Corollary}
\theoremstyle{definition}
\newtheorem{definition}[thm]{Definition}

\newtheorem{prop}{Proposition}[section]

\newtheorem{problem}{Problem}
\newtheorem{theorem}{Theorem}
\newtheorem{lemma}{Lemma}

\newtheorem{exe}{Exercise}
\newtheorem{exa}{Example}

\newtheorem{question}{Question}

\newtheorem{conjecture}{Conjecture}

\newcommand{\blem}{\begin{lemma}}
\newcommand{\elem}{\end{lemma}}
\newcommand{\bexer}{\begin{exe}}
\newcommand{\eexer}{\end{exe}}
\newcommand{\beq}{\begin{eqnarray}}
\newcommand{\eeq}{\end{eqnarray}}
\newcommand{\bthm}{\begin{theorem}}
\newcommand{\ethm}{\end{theorem}}
\newcommand{\beg}{\begin{exa}}
\newcommand{\eeg}{\end{exa}}

\newcommand{\bdefe}{\begin{definition}}
\newcommand{\edefe}{\end{definition}}

\newcommand{\bprop}{\begin{prop}}
\newcommand{\eprop}{\end{prop}}

\newcommand{\bpf}{\begin{proof}}
\newcommand{\epf}{\end{proof}}

\def\be{\begin{equation}}
\def\ee{\end{equation}}
\newtheorem{cl}{Claim}
\newcommand{\bcl}{\begin{cl}}
\newcommand{\ecl}{\end{cl}}

\newcommand{\B}{\mathbb{B}}
\newcommand{\R}{\mathbb{R}}
\newcommand{\C}{\mathbb{C}}

\newcommand{\re}{\mathop{\mathrm{Re}\,}}

\newcommand{\bc}{\mathbf{B}}
\newcommand{\s}{\mathbf{S}}

\newcommand{\bquess}{\begin{questions}}
\newcommand{\equess}{\end{questions}}
\newcommand{\br}{\begin{rem}}
\newcommand{\er}{\end{rem}}
\newcommand{\brs}{\begin{remarks}}
\newcommand{\ers}{\end{remarks}}

\newcommand{\bn}{\begin{nonsec}}
\newcommand{\en}{\end{nonsec}}

\newcommand{\bfig}{\begin{figure}}
\newcommand{\efig}{\end{figure}}
\newcommand{\bstate}{\begin{statement}}
\newcommand{\estate}{\end{statement}}

\newcommand{\bcor}{\begin{cor}}
\newcommand{\ecor}{\end{cor}}
\newcommand{\bprob}{\begin{problem}}
\newcommand{\eprob}{\end{problem}}
\newcommand{\bcon}{\begin{conjecture}}
\newcommand{\econ}{\end{conjecture}}
\newcommand{\bques}{\begin{question}}
\newcommand{\eques}{\end{question}}

\newcommand{\begs}{\begin{examples}}
\newcommand{\eegs}{\end{examples}}

\newcommand{\bdefes}{\begin{definitions}}
\newcommand{\edefes}{\end{definitions}}

\newcommand{\bsol}{\begin{solution}}
\newcommand{\esol}{\end{solution}}

\newcommand{\bexers}{\begin{exercises}}
\newcommand{\eexers}{\end{exercises}}

\newcommand{\ba}{\begin{array}}
\newcommand{\ea}{\end{array}}

\newcommand{\beqq}{\begin{eqnarray*}}
\newcommand{\eeqq}{\end{eqnarray*}}
\newcommand{\bee}{\begin{enumerate}}
\newcommand{\eee}{\end{enumerate}}

\newcommand{\bei}{\begin{itemize}}
\newcommand{\eei}{\end{itemize}}
\newcommand{\bed}{\begin{description}}
\newcommand{\eed}{\end{description}}

\newcommand{\bo}{\begin{obser}}
\newcommand{\eo}{\end{obser}}
\newcommand{\bos}{\begin{obsers}}
\newcommand{\eos}{\end{obsers}}

\newcommand{\IR}{{\mathbb R}}

 \newcommand{\Sp}{{\mathbb S}}

\numberwithin{equation}{section}

\input amssym.def
\input amssym

\begin{document}
\author{M. Mateljevi\'c, N. Mutavd\v{z}i\'c  and  B. Purti\'c }
\title[Schwarz]{$T_\alpha$-harmonic functions in higher dimensions}
\maketitle

\section{Introduction}

This a very rough version  for temporarily use  and it is a small part  of a   project.

In \cite{Liu2004}, Liu and Peng  introduced the following differential operators on $\B^n$ the unit ball in $\R^n$, $n\geq 2$ and $\gamma\in \R$.

\begin{equation}\label{gamaLap}
\Delta_\gamma=(1-|x|^2)\left\{\ \frac{1-|x|^2}{4}\sum_j \frac{\partial^2}{\partial x_j^2} +\gamma \sum_j x_j \frac{\partial}{\partial x_j}+\gamma\left( \frac{n}{2}-1-\gamma \right)  \right\}.
\end{equation}

Through this article we will use more convenient notation

\begin{equation}\label{alfaLap}
T_{\alpha}u(x)=(1-|x|^2)\Delta u(x)+2 \alpha \langle x,\nabla u(x)\rangle   +
(n-2-\alpha) \alpha  u(x).
\end{equation}

Here, $|x|<1$ and $\alpha>-1$. Also, for $x=x(x_1,x_2,\ldots,x_n)\in \mathbb{R}^n$ we use $|x|=\sqrt{x_1^2+x_2^2+\ldots+x_n^2}, \nabla =(\frac{\partial}{\partial x_1},\frac{\partial}{\partial x_2},\ldots,\frac{\partial}{\partial x_n}),\Delta=\frac{\partial^2}{\partial x_1^2}+\frac{\partial^2}{\partial x_2^2}+\ldots+\frac{\partial^2}{\partial x_n^2}.$

The purpose of this paper is to investigate a Dirichlet problem, corresponding to above mentioned PDE. We will specificaly consider non-homogenous boundary value problem. In that purpose explicit form of Green function assosiated to the operator (\ref{alfaLap}) will be calculated, and also, we will present the corresponding representation theorem.

Theorem 1.1. \cite{Liu2004} Let $f\in C^\infty(S_{n-1}),  f \neq 0$. Then the solution $u$ to the Dirichlet
problem
is in $C^\infty(\overline{\mathbb{B}}_n)$ if and only if one of the following occurs:
\begin{itemize}
\item[(H1)] $\gamma$    is a nonnegative integer;

\item[(H2)] the data $f$    has a finite spherical harmonic expansion.

\end{itemize}


These operators appear in a natural way when we transplant the Weinstein equation from the upper half-plane $\{x_n>0\}$ to the unit ball of $\R^n$ via M\"obius transformations. see \cite{Akin,Heinz1} \\

A {\bf complex-ball} counterparts of these operators, the laplacians  on $\bc$, the unit ball of $\C^n$, $n\geq 1$, with $\alpha,\beta\in \C$,

$$\Delta_{\alpha,\beta}=(1-|z|^2)\left\{ \sum_{i,j} (\delta_{ij}-z_i\overline{z}_j) \frac{\partial^2}{\partial z_i \partial \overline{z}_j}+\alpha \sum_{j} z_j\frac{\partial}{\partial z}_j +
\beta \sum_{j} \overline{z}_j\frac{\partial}{\partial \overline{z}_j}-\alpha\beta \right\}.$$

introduced in Geller \cite{Geller} and have been considered by many authors; see, e.g., \cite{Ahern,AhernCas}.\\

\begin{enumerate}
    \item If $\alpha=\beta=0$, $\Delta_{0,0}$ is the invariant laplacian or {\it Bergman laplacian}. The functions annihilated by $\Delta_{0,0}$ are called invariant harmonic functions of $\mathcal{M}$-harmonic, see Rudin \cite[Chapter 4]{Rudin}.
    \item If $\alpha=\beta$, $\Delta_{\alpha,\alpha}$ is the laplacian with respect to the weighted Bergman metric, with weight $(1-|z|^2)^\alpha$.
\end{enumerate}
Thus $(\alpha,\beta)$-harmonic functions can be seen as generalized $\mathcal{M}$-harmonic functions. \\

With $\Delta_{\alpha,\beta}$ there is associated kernel
$$P_{\alpha,\beta}(z,\xi)=c_{\alpha,\beta} \frac{(1-|z|^2)^{n+\alpha+\beta}}{(1-z\overline{\xi})^{n+\alpha}(1-\xi\overline{z})^{n+\beta}}, \quad z\in \bc, \xi \in \s,$$
where
$$c_{\alpha,\beta}=\frac{\Gamma(n+\alpha)\Gamma(n+\beta)}{\Gamma(n)\Gamma(n+\alpha+\beta)}.$$

Now we recall   some considerations from  \cite{Ahern}.
Introduce the fundamental solution for$\Delta_{\alpha,\beta}$, i.e., the function $G_{\alpha,\beta}(z)=g_{\alpha,\beta}(|z|^2)$ which  plays  the role of Green's  function in the classical potential theory. We look for a   radial function which is   $(\alpha,\beta)$-harmonic  away from zero.
They found  $$ G_{\alpha,\beta}(z)= d_{\alpha,\beta} (1- |z|^2)^{n+\alpha + \beta }    F(n+\beta,n+\alpha;n+\alpha + \beta +1;1- |x|^2),$$

with  $$ d_{\alpha,\beta}=\frac{\Gamma (n+\alpha)\Gamma (n+\beta)}{\pi^n\Gamma (n+\alpha + \beta +1 )} .$$

In \cite{Ahern}, for $u\in C^2(\overline{B}^n)$  the following Riesz decomposition formula is obtained
\begin{align} \label{repF}
u(z)&=\int_{\s} P_{\alpha,\beta}(z,\zeta)\varphi(\zeta) d\sigma(\zeta) +\\  \notag +&\int_{B^n} \Delta_{\alpha,\beta} u(\omega) G_{\alpha,\beta}(z,\omega) (1-\overline{z}\omega )^\alpha (1-z\overline{\omega})^\beta (1-|\omega|^2)^{-\alpha-\beta}d\lambda(\omega)
\end{align}
with $ G_{\alpha,\beta}(z,\omega)= g_{\alpha,\beta}(|\varphi_z(\omega)|^2 )=G_{\alpha,\beta}(\varphi_z(\omega))$.

Write the version of this formula with  $r\overline{B}^n)$ instead of    $\overline{B}^n)$.
Strictly speaking, formula (\ref{repF}) has only been obtained for  $u\in  C^2(\overline{B}^n)$, but
it can be seen to hold under more general conditions ??. For instance, it holds if
 $u\in C^2({B}^n) \cap C(\overline{B}^n)$ and

\be
\int_{B^n} |\Delta_{\alpha,\beta} u(\omega)|  \frac{dV(\omega)}{1- |\omega|} < + \infty.
\ee
 $dV$  denotes Lebesgue measure.

 Note that from  definition $G_{\alpha,\beta}$ it follows  that   (i)  $G_{\alpha,\beta}(z)\approx  (1- |z|^2)^{n+\alpha + \beta }$   if   $|z|\rightarrow 1$ and we use (i) to derive
formula (\ref{repF}).
\subsection{hyperbolic harmonic }
Dirchet problem is well understood   for smooth metrics.  For example, one can see Chapter IX of \cite{shen-yau}.
It turns out that this problem for hyperbolic  metric on the unit  ball with boundary data on the unit sphere is very interesting and it  is considered recently in \cite{ChenRas}. Here the metric density  goes to  $\infty $ near the boundary.
Among the other things,  J. Chen, M. Huang, A. Rasila and X. Wang used nice properties of M\"{o}bius transformation and   hyperbolic  Green function of  the unit  ball (described for example  in \cite{ahlMob,stoll}) and integral estimate.
Recall that hyperbolic Laplace operator in the $n$-dimensional hyperbolic ball
$\mathbb{B}^n$ for $n\geq 2$ is defined as\vspace{-3mm}
$$ \Delta_{h}u(x)= (1-|x|^2)^2\Delta u(x)+2(n-2)(1-|x|^2)\sum_{i=1}^{n}
x_{i} \frac{\partial u}{\partial x_{i}}(x).\vspace{-2mm} $$


This can be stated as follows: \\
Let us consider the following Dirichlet boundary problem
\be\label{eqHypDir}\left\{
\begin{array}{ll}
 u(x)=\phi(x),        & \hbox{if} \,\, x\in \mathbb{S}^{n-1}, \\
(\Delta_h)u(x)=\psi, & \hbox{if}  \,\,  x\in \mathbb{B}^n .
\end{array}
\right.\vspace{-2mm}\ee



Set $0\leq r<t\leq 1$ and define \vspace{-2mm}
\be
g(r,t):=\int_r^t \frac{(1-t^2)^{n-2}}{t^{n-1}}d t,\quad g(r)=g(r,1).\vspace{-2mm}
\ee


Then hyperbolic Green function is given by\vspace{-2mm}
\be
G_h(x,y)=g(|T_y x|)=g\left(\frac{|x-y|}{[x,y]}\right).
\ee
The Poisson-Szego kernel $P_h$ for hyperbolic Laplacian  $\Delta_h$ is given by\vspace{-1mm}
\be
P_h (x, t) = \left(\frac{1 - |x|^2|}{|t - x|^2}\right)^{n-1}.\vspace{-1mm}
\ee

Let us define hyperbolic Poisson's integral\vspace{-2mm}
\be
P_h [f](x) =
\int_{\mathbb{S}^{n-1}}P_h (x, t) f(t) d S (t).\vspace{-2mm}
\ee
for $f\in L^1(\mathbb{S}^{n-1})$ and, also, the hyperbolic Green integral\vspace{-2mm}
$$G_{h}[\psi](x)=  \int_{\mathbb{B}^n} G_{h}(x,y) \psi(y) \frac{d V(x)}{(1-|x|^2)^n}, \vspace{-2mm}$$
for appropriate functions $\psi$.

\begin{theorem}\cite{ChenRas}\label{ChenRas1}
Suppose that $u \in  C^2(\mathbb{B}^n , \mathbb{R}^n ) \cap C(\overline{\mathbb{B}^n} , \mathbb{R}^n)$ for  $n\geq 3$ and\vspace{-3mm}
$$\int_{\mathbb{B}^n}(1 {-} |x|^2)^{n{-}1}|\psi(x)| d\tau (x) {\leq}\mu_1, \mbox{ where }\mu_1 > 0\mbox{ is a constant. }\hspace{1cm}\vspace{-2mm}$$
If u satisfies (\ref{eqHypDir}), then\vspace{-2mm}

$$u =
\int_{\mathbb{S}^{n-1}}P_h (x, t) f(t) d\sigma (t) - \int_{\mathbb{B}^n} G_{h}(x,y) \psi(y) \frac{d V(x)}{(1-|x|^2)^n}.$$

\
\end{theorem}

The authors of \cite{ChenRas} proved the above theorem by using
\begin{theorem}\cite[Lemma 3.2]{stoll} If the function $u$ is of the class $C^2(\mathbb{B}^n)$ then
    \begin{equation}
        \label{HypBr}
        u(0)=\int_{\mathbb{S}^{n-1}}u(r\zeta)d S(\zeta)+\int_{r\mathbb{B}^n}g(|x|,r)\Delta_h u(x)\frac{d V(x)}{(1-|x|^2)^n}
    \end{equation}
    for all $0<r<1$.
\end{theorem}

It is important to note that function $g(|x|,r)$ vanishes for $|x|=r$.

In the proof of the main result we will use appropriate analog of the function $g$ in the case of $T_\alpha$ Laplacian.

\subsection{Foundamental solution}








The hypergeometric function is a solution of Euler's hypergeometric differential equation

    $${\displaystyle H_{abc}w=z(1-z){\frac {d^{2}w}{dz^{2}}}+\left[c-(a+b+1)z\right]{\frac {dw}{dz}}-ab\,w=0.} $$

which has three regular singular points: $0,1$ and $\infty$.

Question.  What are solution  of  Euler's hypergeometric differential equation  $H_{abc}w=0$  with singular point at $0$.

Convergence of  Hypergeometric function

In \url{https://en.wikipedia.org/wiki/Hypergeometric_function}, we can find:
Around the point $z = 0$, two independent solutions are, if $c$ is not a non-positive integer,

    $$ {\displaystyle \,_{2}F_{1}(a,b;c;z)}$$

and, on condition that $c$ is not an integer,

   $${\displaystyle z^{1-c}\,_{2}F_{1}(1+a-c,1+b-c;2-c;z)} .$$

   If $c$ is a non-positive integer $1-m$, then the first of these solutions does not exist and must be replaced by ${\displaystyle z^{m}F(a+m,b+m;1+m;z).}$ The second solution does not exist when $c$ is an integer greater than $1$, and is equal to the first solution, or its replacement, when $c$ is any other integer. So when $c$ is an integer, a more complicated expression must be used for a second solution, equal to the first solution multiplied by $\ln(z)$, plus another series in powers of z, involving the digamma function. See Olde Daalhuis (2010) for details.





The generalization of this equation to three arbitrary regular singular points is given by Riemann's differential equation. Any second order linear differential equation with three regular singular points can be converted to the hypergeometric differential equation by a change of variables.
Solutions at the singular points

Solutions to the hypergeometric differential equation are built out of the hypergeometric series $F(a,b;c;z)$. The equation has two linearly independent solutions. At each of the three singular points $0, 1$, $\infty$, there are usually two special solutions of the form $x^s$ times a holomorphic function of $x$, where $s$ is one of the two roots of the indicial equation and $x$ is a local variable vanishing at a regular singular point. This gives $3 \times 2 = 6$ special solutions, as follows.

Around the point $z = 0$, two independent solutions are, if $c$ is not a non-positive integer,

    $$ {\displaystyle \,_{2}F_{1}(a,b;c;z)}$$

and, on condition that $c$ is not an integer,

   $${\displaystyle z^{1-c}\,_{2}F_{1}(1+a-c,1+b-c;2-c;z)} .$$

If $c$ is a non-positive integer $1-m$, then the first of these solutions does not exist and must be replaced by ${\displaystyle z^{m}F(a+m,b+m;1+m;z).}$ The second solution does not exist when $c$ is an integer greater than $1$, and is equal to the first solution, or its replacement, when $c$ is any other integer. So when $c$ is an integer, a more complicated expression must be used for a second solution, equal to the first solution multiplied by $\ln(z)$, plus another series in powers of z, involving the digamma function. See Olde Daalhuis (2010) for details.

If $c$ is not a non-positive integer, $c\neq 0,-1,-2,-3,...$,

    $$ {\displaystyle \,_{2}F_{1}(a,b;c;z)}$$

is a solution of  $ H_{abc}w=0$. Set   $z=1-x$   and  consider   $F^1=  F (a,b;c;1-x)$.  Then $c-(a+b+1)(1-x)=  c-(a+b+1)-  [c-(a+b+1)] x    $.  Check

$ -( c-(a+b+1)-  [c-(a+b+1)] x)F_1'(1-x) = ( (a+b+1) -c  +  [c-(a+b+1)] x)F_1'(1-x) $,  $c_1= (a+b+1) -c$,  $(a+b+1) -c=a_1+b_1+1$,  $(a+b -c=a_1+b_1$,  $a_1 b_1= ab$

Solve $(a+b -c=a_1+b_1$,  $a_1 b_1= ab$ in  $a_1,  b_1$.   $F_1= F (a_1,b_1;c_1;1-x)$  is a  solution of   $ H_{a_1b_1c_1}w=0$  and $F_1(0)=\infty$ .

Foundamental solution for equation $T_{\alpha}u(x)=0$:

For extension of Euler's first and second transformations, see Rathie \& Paris (2007) and Rakha \& Rathie (2011). It can also be written as linear combination.

\bcl\label{sol1}
Around $z = 1$, if   $c- a - b$ is not an integer, one has two
independent solutions  $abc$-the hypergeometric differential equation:

$ {\displaystyle  X_1(z)= \,_{2}F_{1}(a,b;1+a+b-c;1-z)}$

and

${\displaystyle X_2(z)=(1-z)^{c-a-b}\;_{2}F_{1}(c-a,c-b;1+c-a-b;1-z)}$.
\ecl
Note if   $c= a + b$, then $X_1(z)= X_2(z)= \,_{2}F_{1}(a,b;1;1-z)$.

We consider  foundamental solution for equation $T_{\alpha}u(x)=0$:
We look for a radial function $G_{\alpha,n}(x)=g(|x|^2)$ so that  it
is fulfilled
$T_{\alpha}G_{\alpha,n}(x)=0$.
After short calculation we obtain equation for $0<x<1$.

$$x(1-x)g''(x)+\left(\frac{n}{2}-\left(\frac{n}{2}-\alpha\right)x\right)g'(x)+\frac{(n-2-\alpha)\alpha}{4}g(x)=0.$$
This is hypergeometric equation for

$a=-\frac{\alpha}{2}$,  $b=\frac{n-2-\alpha}{2}$ and
$c=\frac{n}{2}$.(see Erdelyi: "Higher
transcedental functions", vol 1).

Here for $ \alpha\neq -2,-3,-4,...$ (in particular $ \alpha  >- 1$),
with  Claim \ref{sol1}, $X_2(z)$ is solution with
$c-a=\frac{n+\alpha}{2},c-b=\frac{2+\alpha}{2},1+c-a-b=\alpha+2$,
$c-a-b=\alpha+1$.
If  $d=c- a - b$  is a negative    integer  it seems  $ {\displaystyle
X_1(z)= \,_{2}F_{1}(a,b;1+a+b-c;1-z)}$ is a  solution.

Since $a=-\frac{\alpha}{2}$,
$b=\frac{n-2-\alpha}{2}$ and $c=\frac{n}{2}$, we find
$1+a+b-c=-\alpha$;   then by Claim \ref{sol1},
$X_1(z)=\,_{2}F_{1}(a,b;-\alpha;1-z)$ is a solution.

Since $1+a+b-c>a +b$ iff  $1<c$. Since $c=-\alpha$,  $1<c$   iff $
\alpha  >- 1$.  Here $X_1(0)$ finite it is not  good candidate for
Green.

  If $a>c-b>0$, $b>0$, then   $F(a,b,c,x)\sim  (1-x)^{c-a-b}$.  In
particular  $F(a,b,c,x)\rightarrow \infty$ if $x\rightarrow 1_-$.  See
Flett  \cite{Flett}.

Hence  $Y(x):=
F\left(\frac{n+\alpha}{2},\frac{2+\alpha}{2},\alpha+2;1-|x|^2\right)
\sim |x|^{2-n} $   if $x\rightarrow 0$.

Here $c-a-b= 1- n/2$. Thus for $n>2$, since  $1-( 1-|x|^2)=|x|^2 $,
$Y(x)\sim (|x|^2)^{ 1- n/2}$ and therefore   $Y(x) \sim |x|^{2-n} $  if
$x\rightarrow 0$.

Thus $X_2$  has  singularity  required for Green function at $x=0$.




Let $R$ be the radial derivative, $Ru(x)=\sum\limits_{i=1}^{n}\frac{\partial{u}}{\partial x_i}x_i=\langle x,\nabla u(x)\rangle $.
\begin{thm}
Let $u,v \in C^2(\mathbb{B}^n)$, $0<r<1$, $\alpha>-1$. Then
\begin{equation}\label{a-GrFla}
r^{n-2}(1-r^2)^{-\alpha}\int\limits_{\Sp^{n-1}}(uRv-vRu)(r\zeta)dS(\zeta)=
\int\limits_{r\B^n}(u(x)T_{\alpha}v(x)-v(x)T_{\alpha}u(x))(1-|x|^2)^{-\alpha-1}dV(x)
\end{equation}
\end{thm}
\begin{proof}
To simplify notations we use
$$ \rho=1-||x||^2, \ w_i=(-1)^{i+1}dx_1 \wedge \ldots \wedge \widehat{dx_i}\wedge \ldots \wedge dx_n, \ w=dx_1\wedge dx_2\wedge \ldots \wedge dx_n=dV. $$
Let $W=\rho^{-\alpha}v\sum\limits_{i=1}^{n}\frac{\partial u}{\partial x_i}w_i$ be form defined in $r\B^n$. Using the general Stokes formula (conditions are satisfied) we obtain
$$\int\limits_{r\B^n}dW=\int\limits_{r\Sp^{n-1}}W.$$
Let us now find $dW$.
Notice first that $d{\rho}=-2\sum\limits_{i=1}^{n}x_i dx_i$ , $dw_i=0$, $dx_i\wedge w_i=w$ and $dx_j\wedge w_i=0$ for $i\neq j$.
Now
\begin{eqnarray*}dW&=&d(\rho ^{-\alpha})v\sum\limits_{i=1}^{n}\frac{\partial u}{x_i}w_i+\rho^{-\alpha}d\left(v\sum\limits_{i=1}^{n}\frac{\partial u}{\partial x_i}w_i\right) \\
&=& (-\alpha)\rho^{-\alpha-1}\left(-2\sum_{i=1}^{n}x_i dx_i\right)v\sum\limits_{i=1}^{n}\frac{\partial u}{x_i}w_i+\rho^{-\alpha}\left(\sum_{i=1}^{n}\frac{\partial v}{x_i}dx_i\sum_{i=1}^{n}\frac{\partial u}{\partial x_i}w_i+v\sum_{i=1}^{n}d\left(\frac{\partial u}{\partial x_i}\right)w_i\right) \\
&=& 2\alpha v \rho^{-\alpha-1}\sum\limits_{i=1}^{n}x_i\frac{\partial u}{\partial x_i}w+\rho^{-\alpha}\left(\sum\limits_{i=1}^{n}\frac{\partial v}{\partial x_i}\frac{\partial u}{\partial x_i}w+v\sum\limits_{i=1}^{n}\left(\frac{{\partial}^2 u}{{\partial x_i}^2}w\right)\right) \\
&=&  2\alpha v \rho^{-\alpha-1}\sum\limits_{i=1}^{n}x_i\frac{\partial u}{\partial x_i}w+\rho^{-\alpha}
\left(\sum\limits_{i=1}^{n}\frac{\partial v}{\partial x_i}\frac{\partial u}{\partial x_i}+v\Delta{u}\right)w\\
&=& \rho^{-\alpha-1}v\left(2\alpha R{u}+\rho \Delta{u}\right)w+\rho^{-\alpha}\sum\limits_{i=1}^{n}\frac{\partial v}{\partial x_i}\frac{\partial u}{\partial x_i}w\\
\end{eqnarray*}
Now from general Stoke's formula we have
$$\int\limits_{r\B^n}\left(\rho^{-\alpha-1}v\left(2\alpha R{u}+\rho \Delta{u}\right)+\rho^{-\alpha}\sum\limits_{i=1}^{n}\frac{\partial v}{\partial x_i}\frac{\partial u}{\partial x_i}\right)w=\int\limits_{r\Sp^{n-1}}\rho^{-\alpha}v\sum\limits_{i=1}^{n}\frac{\partial u}{\partial x_i}w_i.$$
Repeating this proces with $u$ and $v$ interchanged we obtain
$$\int\limits_{r\B^n}\left(\rho^{-\alpha-1}u\left(2\alpha R{v}+\rho \Delta{v}\right)+\rho^{-\alpha}\sum\limits_{i=1}^{n}\frac{\partial u}{\partial x_i}\frac{\partial v}{\partial x_i}\right)w=\int\limits_{r\Sp^{n-1}}\rho^{-\alpha}u\sum\limits_{i=1}^{n}\frac{\partial v}{\partial x_i}w_i.$$
By subtracting that two equations we obtain
$$\int\limits_{r\B^n}\rho^{-\alpha-1}\left( u(2\alpha R{v}+\rho \Delta{v})-v(2\alpha R{u}+\rho \Delta{u})\right) w=\int\limits_{r\Sp^{n-1}}\rho^{-\alpha}\left(u\sum\limits_{i=1}^{n}\frac{\partial v}{\partial x_i}w_i-v\sum\limits_{i=1}^{n}\frac{\partial u}{\partial x_i}w_i\right).$$
Let's remember that $$T_{\alpha}u(x)=(1-|x|^2)\Delta u(x)+2 \alpha \langle x,\nabla u(x)\rangle   +
(n-2-\alpha) \alpha  u(x).$$
Using notations introduced here, equation has the form
$$T_{\alpha}u(x)=\rho\Delta u(x)+2 \alpha Ru(x)+
(n-2-\alpha) \alpha  u(x) $$
and now we have equation
$$uT_{\alpha}v-vT_{\alpha}u=u(\rho \Delta{v}+2\alpha Rv)-v(\rho\Delta{u}+2\alpha Ru).$$
Finally we obtain
\begin{eqnarray*}
\int\limits_{r\B^n}\left(uT_{\alpha}v-vT_{\alpha}u \right) \rho^{-\alpha-1} w&=&(1-r^2)^{-\alpha}\int\limits_{r\Sp^{n-1}}\left(u\sum\limits_{i=1}^{n}\frac{\partial v}{\partial x_i}w_i-v\sum\limits_{i=1}^{n}\frac{\partial u}{\partial x_i}w_i\right) \\
&=& (1-r^2)^{-\alpha}\int\limits_{r\Sp^{n-1}}\left(u\sum\limits_{i=1}^{n}\frac{\partial v}{\partial x_i}x_i-v\sum\limits_{i=1}^{n}\frac{\partial u}{\partial x_i}x_i\right)\frac{1}{r}dS(x) \\
&=& (1-r^2)^{-\alpha}\int\limits_{r\Sp^{n-1}}\left(uRv-vRu\right)(x)\frac{1}{r}dS(x) \\
&=& (1-r^2)^{-\alpha}r^{n-2}\int\limits_{\Sp^{n-1}}\left(uRv-vRu\right)(r\zeta)dS(\zeta).
\end{eqnarray*}
(Unit normal is $\frac{x}{||x||}=\frac{x}{r}$, so $\frac{x_i}{r} dS(x)=w_i$. See Zorich II, equation 13.21.
Change of variables gives the last equation.)

\end{proof}

%
%
%

$$\frac{d}{dx}\left[x^{c-1}F(a,b;c,x)\right]=(c-1)x^{c-2}F(a,b;c-1;x)$$

\begin{align}\label{u_6}
   u_6 &=(1-x)^{c-a-b}F(c-a,c-b;c+1-a-b;1-x)=\\
   &=x^{1-c}(1-x)^{c-a-b}F(1-a,1-b;c+1-a-b;1-x)
\end{align}


$$F(a,b;c;1)=\frac{\Gamma(c)\Gamma(c-a-b)}{\Gamma(c-a)\Gamma(c-b)}, c\neq 0,-1,-2,\ldots, Re c> Re(a+b).$$

$$F(1-a,1-b;c+1-a-b;1)=\frac{\Gamma(c-1)\Gamma(c+1-a-b)}{\Gamma(c-a)\Gamma(c-b)}$$

$$a=-\frac{\alpha}{2},b=\frac{n-2-\alpha}{2},c=\frac{n}{2}$$

$$c-a=\frac{\alpha+n}{2},c-b=\frac{\alpha+n}{2}, c+1-a-b=\alpha+2$$

$$h(x)=(1-x)^{\alpha+1}F\left(\frac{\alpha+n}{2},\frac{\alpha+2}{2};\alpha+2;1-x\right)$$

$$F\left(\frac{\alpha+n}{2},\frac{\alpha+2}{2};\alpha+2;x\right)\sim  (1-x)^{1-\frac{n}{2}}\frac{\Gamma(\alpha+2)\Gamma(\frac{n}{2}-1)}{\Gamma(\frac{\alpha+n}{2})\Gamma(\frac{\alpha+2}{2})}, x\to 1^-$$

$$G_\alpha(x)=d_\alpha (1-|x|^2)^{\alpha+1}F\left(\frac{\alpha+n}{2},\frac{\alpha+2}{2};\alpha+2;1-|x|^2\right)$$

$$R G_\alpha(x)=-2d_\alpha (\alpha+1) |x|^2(1-|x|^2)^\alpha F\left(\frac{\alpha+n}{2},\frac{\alpha+2}{2};\alpha+1;1-|x|^2\right)$$

$$a=-\frac{\alpha}{2}+1,b=\frac{n-2-\alpha}{2}+1,c=\frac{n}{2}+1$$

$$F\left(\frac{\alpha+n}{2},\frac{\alpha+2}{2};\alpha+1;x\right)=(1-x)^{-\frac{n}{2}}F\left(\frac{\alpha}{2},\frac{\alpha+2-n}{2};\alpha+1;x\right)$$

$$F\left(\frac{\alpha}{2},\frac{\alpha+2-n}{2};\alpha+1;1\right)=\frac{\Gamma(\alpha+1)\Gamma(\frac{n}{2})}{\Gamma(\frac{\alpha+n}{2})\Gamma(\frac{\alpha+2}{2})}$$

$$R G_\alpha (x)\sim -2 d_\alpha (\alpha+1)|x|^{2-n}\frac{\Gamma(\alpha+1)\Gamma(\frac{n}{2})}{\Gamma(\frac{\alpha+n}{2})\Gamma(\frac{\alpha+2}{2})}, x\to 0$$

$$r^{n-2}(1-r^2)^{-\alpha}\int\limits_{\Sp^{n-1}}\left\{u(r\zeta)R G_\alpha(r\zeta)-G_\alpha(r\zeta)Ru(r\zeta)\right\}dS(\zeta)-$$

$$-\epsilon^{n-2}(1-\epsilon^2)^{-\alpha}\int\limits_{\Sp^{n-1}}\left\{u(\epsilon\zeta)R G_\alpha(\epsilon\zeta)-G_\alpha(\epsilon\zeta)Ru(\epsilon\zeta)\right\}dS(\zeta)=$$

\begin{equation}\label{GrinBr}
=\int\limits_{r\B^n\setminus \epsilon\B^n}(T_{\alpha}u(x))G_\alpha(x)(1-|x|^2)^{-\alpha-1}dV(x)
\end{equation}

$$d_\alpha=-\frac12\frac{\Gamma(\frac{\alpha+n}{2})\Gamma(\frac{\alpha+2}{2})}{\Gamma(\alpha+1)\Gamma(\frac{n}{2})}, c_\alpha=-2\frac{\Gamma(\alpha+1)\Gamma(\frac{n}{2})}{\Gamma(\frac{\alpha+n}{2})\Gamma(\frac{\alpha+2}{2})}$$

Set  $\psi=T_{\alpha}u(x)$  and suppose that  $\psi$ is integrable on  $\B^n$.

Since  $G_\alpha(x)\preceq (1-|x|^2)^{\alpha+1}$ if $|x|\rightarrow 1$, then  $G_\alpha(x)(1-|x|^2)^{-\alpha-1}$  is bounded for $|x|\rightarrow 1$.

\begin{equation}\label{DirRes0}
u(0)=c_\alpha \int\limits_{\Sp^{n-1}}u(\zeta)dS(\zeta)+\int\limits_{\B^n}(T_{\alpha}u(x))G_\alpha(x)(1-|x|^2)^{-\alpha-1}dV(x)
\end{equation}

In \cite{Liu2004}, Liu and Peng  introduced the following differential operators on $\B^n$ the unit ball in $\R^n$, $n\geq 2$ and $\gamma\in \R$.

$$\Delta_\gamma=(1-|x|^2)\left\{\ \frac{1-|x|^2}{4}\sum_j \frac{\partial^2}{\partial x_j^2} +\gamma \sum_j x_j \frac{\partial}{\partial x_j}+\gamma\left( \frac{n}{2}-1-\gamma \right)  \right\}.$$

\begin{equation*}
    \varphi_a(x)=\frac{|x-a|^2a-(1-|a|^2)(x-a)}{[x,a]^2},
\end{equation*}
gde je
$$[x,a]=|x||x^*-a|=|a||x-a^*|\quad \mbox{i}\quad [x,a]^2=1+|x|^2|a|^2-2xa.$$

by a simple computation we get $\varphi_x(0)=x$ and therefeore $\varphi_{x}=\varphi_x^{-1}$.

Derivative of the differentiable function $f$ which maps one open set in $\mathbb{R}^n$ to another can be represented with Jacobian matrix
$$
f'(x)\quad \mbox{ili}\quad Df(x),
$$
with elements
\begin{equation*}
    f'(x)_{ij}=\frac{\partial f_i}{\partial x_j}=D_jf_i(x).
\end{equation*}

We define the longest extension at the point $x$ as $|f'(x)|=\max\limits_{|h|=1}|f'(x)h|$, where $h$ represents a vector in $\mathbb{R}^n$. Also, we say that the square matrix $M$ is conformal if the value of $|Mh|$ is a non-negative real number independent of the unit vector $h$. It can be checked that $\gamma'(x)$ is conformal matrix for any $\gamma\in\widehat{M}(\mathbb{R}^n)$. \smallskip






$$
1-|\varphi_x(y)|^2=\frac{(1-|x|^2)(1-|y|^2)}{[x,y]^2}, |\varphi_x'(y)|=\frac{1-|x|^2}{[x,y]^2}, [x,\varphi_x(y)]=\frac{1-|x|^2}{[x,y]}
$$

In \cite{Liu2004} can be found

$$\Delta_\gamma\left\{|\varphi_x'(y)|^{(n-2-2\gamma)/2}u(\varphi_x(y))\right\}=|\varphi_x'(y)|^{(n-2-2\gamma)/2}(\Delta_\gamma u)(\varphi_x(y))$$

We easily check that $T_\alpha u(x)=\frac{4}{1-|x|^2}\Delta_{\alpha/2}u(x)$ holds true. This gives us formula

\begin{equation} \label{InvLapFor}
    T_\alpha\left\{ \frac{1}{[x,y]^{n-2-\alpha}}u(\varphi_x(y))\right\}= \frac{1-|x|^2}{[x,y]^{n-\alpha}}(T_\alpha u)(\varphi_x(y))
\end{equation}

When applying formula (\ref{DirRes0}) to the function $v(y)= \frac{1}{[x,y]^{n-2-\alpha}}u(\varphi_x(y))$ and after change of variables $z=\varphi_x(y)$ we get

\begin{equation}\label{DirRes}
u(x)=c_\alpha \int\limits_{\Sp^{n-1}}u(\zeta)\frac{(1-|x|^2)^{1+\alpha}}{|x-\zeta|^{n+\alpha}}dS(\zeta)+\int\limits_{\B^n}(T_{\alpha}u(y))G_\alpha(x,y)(1-|y|^2)^{-\alpha-1} d V(y).
\end{equation}

where $G_\alpha(x,y)=|\varphi'_x(y)|^{(n-2-\alpha)/2}G_\alpha(\varphi_x(y))$. Here is important to notice that $G_\alpha(y)=G_\alpha(0,y)$.







Let us consider the following Dirichlet boundary problem  for $T_{\alpha}$  laplacian.   \vspace{-2mm}
\be\label{eqHypDir}\left\{
\begin{array}{ll}
 u(x)=\phi(x),        & \hbox{if} \,\, x\in \mathbb{S}^{n-1}, \\
T_\alpha u(x)=\psi(x), & \hbox{if}  \,\,  x\in \mathbb{B}^n .
\end{array}
\right.\ee

Define
\beq
P_{\alpha}[u](x)=c_\alpha \int\limits_{\Sp^{n-1}}u(\zeta)\frac{(1-|x|^2)^{1+\alpha}}{|x-\zeta|^{n+\alpha}}d S(\zeta), \\
G_{\alpha}[\psi](x)=\int\limits_{\B^n}\psi(y)G_\alpha(x,y)(1-|y|^2)^{-\alpha-1} d V(y).
\eeq
\begin{thm}\label{DirOpsti} Let $n\geq 3$ and $u\in
C^{2}(\mathbb{B}^{n},\IR^n) \cap C(\overline{\mathbb{B}^{n}},\IR^n )$ is a solution to  Dirichlet boundary problem for $T_{\alpha}$  harmonic functions, then it has a
representation\vspace{-2mm}
\begin{equation}\label{RepHyp}
u=P_{\alpha}[\phi]+G_{\alpha}[\psi],\vspace{-2mm}
\end{equation}
provided that
\vspace{-2mm}
$$u\mid_{\mathbb{S}^{n-1}}=\phi\quad \mbox{and}\quad
\int_{\mathbb{B}^{n}} |\psi(x)|\,d V(x)\leq C_1<\infty.\vspace{-1mm}$$
\end{thm}
\begin{proof}
After changing $v(x)$ from the formula (\ref{a-GrFla}) with $g_\alpha(x,r)=G_\alpha(x)-G_\alpha(r)$, we can establish a version of formula (\ref{HypBr}) for $T_\alpha$ Laplacian, which can be written as:

\begin{align}
    \label{TaBr}
    u(0)&=d_\alpha r^nF\left(\frac{\alpha+n}{2},\frac{\alpha+2}{2};\alpha+1;1-|x|^2\right)\int_{\mathbb{S}^{n-1}}u(r\zeta)d S(\zeta)\\ \notag
    &+ \int_{r\mathbb{B}^n}T_\alpha u(x)\{G_\alpha(x)-G_\alpha(r)\}(1-|x|^2)^{-\alpha-1}d V(x)\\ \notag
    &+(n-2-\alpha)\alpha G_\alpha(r)\int_{r\mathbb{B}^n} u(x)(1-|x|^2)^{-\alpha-1}d V(x).
\end{align}
Define $M_\psi(r):=\sup\limits_{|x|<r}|\psi(x)|.$ Let's check if
$$\int_{\mathbb{B}^n} G_\alpha(|x|)|\psi(x)|(1-|x|^2)^{-\alpha-1}d V(x)\leq C_3:=C_3(C_1,M_{\psi G_\alpha}(1/2))<\infty,$$
for such $\psi$ as in the statement of this theorem.
First, we write
\begin{align*}
     \int_{\mathbb{B}^n} G_\alpha(|x|)|\psi(x)(1-|x|^2)^{-\alpha-1}d V(x)&= \int\limits_{\frac12\mathbb{B}^n} G_\alpha(|x|)|\psi(x)(1-|x|^2)^{-\alpha-1}d V(x)\\
    &+ \int\limits_{\mathbb{B}^n\setminus(\frac12\mathbb{B}^n)} G_\alpha(|x|)|\psi(x)(1-|x|^2)^{-\alpha-1}d V(x).
\end{align*}

Now, it is imidiate that
$$\int_{\frac12\mathbb{B}^n} G_\alpha(|x|)|\psi(x)|(1-|x|^2)^{-\alpha-1}d V(x)\leq \left(\frac{3}{4}\right)^{-\alpha-1}M_{\psi G_\alpha}(1/2).$$
At the section ($15.4(ii)$) of \cite{dlmf} we can find that
$$\lim\limits_{z\to 1^-} \frac{F(a,b;c;z)}{(1-z)^{c-a-b}}=\frac{\Gamma(c)\Gamma(a+b-c)}{\Gamma(a)\Gamma(b)},$$
when $\re c<\re(a+b)$. This formula gives us that $\lim\limits_{s\to 0^+}k_\alpha(s)=\frac{\Gamma(\alpha+2)\Gamma((n-2)/2)}{\Gamma((\alpha+n)/2)\Gamma((\alpha+2)/2)},$ where $$k_\alpha(|x|)=|x|^{n-2}F\left(\frac{\alpha+n}{2},\frac{\alpha+2}{2};\alpha+2;1-|x|^2\right).$$
Having that in mind, we get
\begin{equation*}
h_\alpha(|x|):=G_\alpha(|x|)(1-|x|^2)^{-\alpha-1}=d_\alpha |x|^{2-n} k_\alpha(|x|).
\end{equation*}
Since $h_\alpha$ is a continuous function on $[0,1)$, along with the fact that $\lim\limits_{s\to 1}h_\alpha(s)=d_\alpha$ we have that
$$\int\limits_{\mathbb{B}^n\setminus(\frac12\mathbb{B}^n)} G_\alpha(|x|)|\psi(x)(1-|x|^2)^{-\alpha-1}d V(x)\leq C_1 M_{h_\alpha}(1/2).$$
\end{proof}







\begin{thm}\cite{mss}
Suppose  that $f:\mathbb{S}^{n-1}\rightarrow \mathbb{R}^n$  is locally
Lipschitz {\rm (Lip-$1$)}   at $x_0 \in \mathbb{S}$,  $f\in
L^\infty(\mathbb{S}^{n-1})$   and $\,h\,=P[f]$ is a Euclidean harmonic
mapping from
$\mathbb{B}^n$.\\
Then
\begin{itemize}
\item[{\rm S1)}]  $$|h'(r x_0)T|\leq M  \eqno (2')$$

for every  $0\leq r <1$ and unit vector  $T$ which is tangent on
$\mathbb{S}^{n-1}_r$  at  $r x_0$,  where $M$ depends only on $n$,
$|f|_\infty$  and  $Lf (x_0)$.

If we suppose in addition that  $h$ is  K-quasiregular (shortly K-qr)  mapping  along  $[o,x_0)$,  then\\
  \item[{\rm S2)}]   $$|h'(r x_0)|\leq K \,M  \eqno (1')$$
  for every  $0\leq r <1$.
\end{itemize}
\end{thm}

Here we only  announce  the following results:
\begin{thm}
Suppose  that $f:\mathbb{S}^{n-1}\rightarrow \mathbb{R}^n$  is locally
Lipschitz {\rm (Lip-$1$)}   at $x_0 \in \mathbb{S}$,  $f\in
L^\infty(\mathbb{S}^{n-1})$   and $\,h\,=P_\alpha[f]$, $\alpha >0$,  is a  $T_{\alpha}$  harmonic
mapping from
$\mathbb{B}^n$.\\
Then

\begin{itemize}
  \item[{\rm S3)}]   $$|h'(r x_0)|\leq  C \eqno (2')$$
  for every  $0\leq r <1$.
\end{itemize}
\end{thm}

\begin{thm}
Suppose  that $f:\mathbb{S}^{n-1}\rightarrow \mathbb{R}^n$  is locally
$\beta$- H\"{o}llder  at $x_0 \in \mathbb{S}$, $0< \beta \leq 1$, $f\in
L^\infty(\mathbb{S}^{n-1})$   and $\,h\,=P_\alpha[f]$, $\alpha >0$,  is a  $T_{\alpha}$  harmonic
mapping from
$\mathbb{B}^n$.\\
Then

\begin{itemize}
  \item[{\rm S4)}]   $$|h'(r x_0)|\leq  C (1-r)^{\beta-1}  \eqno (3')$$
  for every  $0\leq r <1$.
\end{itemize}
In particular if   $\beta=1$,  $|h'(r x_0)|\leq  C$    for every  $0\leq r <1$.

If $f$  is Lip on $\mathbb{S}^{n-1}$, then  $h$  is Lip on  $\mathbb{B}^{n})$.
\end{thm}

For convenient of the reader we  prove:
\begin{prop}[\cite{mss,ArsMat}]\label{prop2a}
 Suppose  that $0 <\alpha<1$  and $x=re_n$, $0<r<1$. Then
$$I_\alpha(r e_n)=: \int_{\mathbb{S}^{n-1}} \frac{|e_n -t|^\alpha}{|x-t|^n}
d\sigma(t)\ \leq  \frac{c_{\alpha, n}}{(1-r)^{1-\alpha}}.
$$
\end{prop}

\begin{proof}
Since the integral is a continuous function of $0 \leq r < 1$, it suffices to prove the estimate under additional assumption
$1/2 \leq r < 1$.
The integrand depends only on the angle $\theta = \angle (t, e_n)$ so we can use integration in
polar coordinates on the sphere $\mathbb S^{n-1}$. This gives
\begin{eqnarray}\label{eqH4}
I_\alpha(r e_n) \leq c_n \int_0^\pi \frac {|\theta|^{n-2}
|\theta|^{\alpha}} {((1-r)^2+\frac{4r}{\pi^2}\theta^2)^{n/2}}\,d\theta < \\
\label{eqH5} c_n \int_0^{\infty}\frac {\theta^{\alpha +n-2}}
{\left((1-r)^2+\frac{4r}{\pi^2}\,
\theta^2\right)^{n/2}}\,d\theta\,.
\end{eqnarray}
Next using  $(1+\frac{4r}{\pi^2} u^2)^{-1} \leq C (1+u^2)^{-1}$  for  $\frac 12 \leq r <1$  and  a  change of variable  $\theta=(1-r) u$, we find
\begin{equation}
I_\alpha(r e_n) \leq C {(1-r)^{\alpha -1
}}\int_0^{\infty}\frac{u^{\alpha +n-2}}{(1+u^2)^{n/2}}\,d u\,.
\end{equation}





Since the above improper integral is convergent the proof is completed.
\end{proof}

Using similar  approach  if  $\omega$ is a majorant ? XXX one  can prove

\begin{prop}\label{Prop3}\cite{ArsMat} XXX
$$I_\omega(r e_n)=: \int_{\mathbb{S}^{n-1}} \frac{\omega(|e_n -t|)}{|x-t|^n}
d\sigma(t)\leq c\cdot \frac {\omega(1-r)} {1-r}.
$$
\end{prop}
\begin{prop}\label{prop3a}
 Suppose  that $0 <\alpha,  0< \beta \leq 1$  and $x=re_n$, $0<r<1$. Then
$$J_\alpha^\beta(r e_n)=: \int_{\mathbb{S}^{n-1}} \frac{|e_n -t|^{\beta}}{|x-t|^{n+\alpha}}
d\sigma(t)\ \leq  \frac{c_{\alpha, n}}{(1-r)^{\beta -1-\alpha}}.
$$
and
$$I_\alpha^\beta(r e_n) = (1-r)^\alpha J_\alpha^\beta(r e_n)  \preceq (1-r)^{\beta-1 }$$
\end{prop}

\begin{lemma}
    Let $D(r,\rho)=\int\limits_{\Sp^{n-1}}\frac{d S(\xi)}{[x,\rho\xi]^s}, |x|=r, 1/2<\rho,r<1.$ Then
    \begin{equation}\label{ocrna1}
        D(r,\rho)\sim
        \begin{cases}
            \frac{1}{(1-\rho)^{s-n+1}}, & s>n-1\\
            \log(1-\rho), & s=n-1,\\
            const, & s<n-1
        \end{cases}
    \rho\to 1^-.
    \end{equation}
    \end{lemma}
    If  $u\in C^2(\mathbb{B}^{n})$  and  $T_{\alpha}u=0$  on  $\mathbb{B}^{n}$ we call  $u$   $T_{\alpha}$  harmonic.


Since

\section{App}

In \cite{ChenRas} the following result is also established:

\begin{theorem}\label{ChenRasThm}\cite[Theorem 1.2]{ChenRas}
  Let $n\geq 3$. Suppose that\vspace{-2mm}
  \begin{itemize}
    \item[(1)] $u\in C^2(\mathbb{B}^n,\mathbb{R}^n)\cap C(\overline{\mathbb{B}^n},\mathbb{R}^n)$ is of the form (\ref{RepHyp});

    \item[(2)] there is a constant $L\geq 0$ such that $|\phi(\xi)-\phi(\eta)|\leq L|\xi-\eta|$ for all $\xi,\eta\in\mathbb{S}^{n-1}$;

    \item[(3)] there is a constant $M\geq 0$ such that $|\psi(x)|\leq M(1-|x|^2)$ for all $x\in\mathbb{B}^n$.
  \end{itemize}
  Then, there is a constant $N=N(n,L,M)$ such that for $x,y\in\mathbb{B}^n$,
  $$|u(x)-u(y)|\leq N|x-y|,\vspace{-2mm}$$
  where the notation $N=N(n,L,M)$ means that the constant $N$ depends only on the quantities $n, L$ and $M$.
\end{theorem}
In fact, they prove more general result:
\begin{itemize}
\item[(A)] If function $\phi$ satisfies $(1)$ then $\Phi=P_h[\phi]$ is Lipshitz and

\item[(B)] if function $\psi$ satisfies $(3)$ then $\Psi=G_h[\psi]$ is Lipshitz.
\end{itemize}
For further results  see \cite{MaMu}.

Here   denote
$$ I_\alpha(r)=\frac{1}{2\pi} \int_0^{2\pi} \frac{(1-|z|^2)^\alpha}{|1-\overline{z} e^{it}|^{\alpha+1}} dt,$$
where $r=|z|.$

In particular, $$ I_0(r)=\frac{1}{2\pi} \int_0^{2\pi} \frac{1}{|1-\overline{z} e^{it}|} dt.$$

Using
$A= |1-re^{it}|=\sqrt{(1-r)^2 +4r \sin^2(t/2)}$,  $\sin(t/2)\geq t/\pi $  and   change of variables   $t=(1-r)u$  one can show that
$ I_0(r)\approx -\log(1-r)$.

Set $B=  ((1-r)^2 + c (1-r)^2u^2)^{s/2}=(1-r)^s  ( 1 +cu^2)^{s/2} $and $s=\alpha+1$. Since    $dt=(1-r)du$,  $s-(s-1)-1$,  and  $s-(s-1)-1=0$,  $ I_\alpha(r)\approx C(r)$, where

 $C(r)=\int_0^{a(r)} 1 +cu^2)^{-s/2}du$  and  $a(r)=2\pi (1-r)^{-1}$. If  $s\leq 1$,    $C(r)\rightarrow  \infty$ and it is bounded if    $s> 1$.
Note that  $I_\alpha(r)$  is bounded for  $\alpha >0$  and therefore $\alpha$-harmonic     functions for $\alpha >0$   and   harmonic   functions have some different properties.
Here   	$I_{\alpha} (r)\approx (1-r^2)^\alpha$ ?  If $\alpha \leq -1$   then  $K_\alpha[1](z) \rightarrow \infty $ if $|z|\rightarrow 1$.
\emph{Dirichlet boundary value problem} of functions $1$  no  has solution if  $\alpha \leq -1$.

\end{document}